\documentclass[11pt]{amsart}

\usepackage{amssymb,amsmath}

\usepackage{enumerate}
\usepackage[a4paper,left=3cm,right=3cm,top=2.5cm,bottom=2.5cm]{geometry}
\usepackage{hyperref,color}

\newtheorem{theore}{Theorem}[section]
\newtheorem{corollar}[theore]{Corollary}
\newtheorem{lemm}[theore]{Lemma}
\theoremstyle{remark}
\newtheorem{exampl}{Example}[section]
\newtheorem{remar}[theore]{Remark}
\newtheorem{assumptio}{Assumption}[section]

\newcommand{\jump}{\mathcal{P}}

\newcommand{\Leb}{\mathrm{Leb}}
\newcommand{\A}{\mathcal{A}}
\newcommand{\B}{\mathcal{B}}
\newcommand{\G}{\mathcal{G}}
\newcommand{\BS}{\mathfrak{X}}
\newcommand{\Ker}{\mathrm{ker}}
\newcommand{\Tr}{\mathrm{Tr}}
\newcommand{\x}{\mathrm{x}}
\newcommand{\z}{\mathrm{v}}

\begin{document}

\title{Transport equations and perturbations of boundary conditions}

\author{Marta Tyran-Kami\'nska}
\thanks{This research was supported by the Polish NCN grant  2017/27/B/ST1/00100.}

\address{Institute of Mathematics, University of Silesia in Katowice, Bankowa 14, 40-007
Katowice, Poland}
\email{mtyran@us.edu.pl}

\begin{abstract}
We provide a new perturbation theorem for substochastic semigroups on abstract AL spaces extending Kato's perturbation theorem to non-densely defined operators. We show how it can be applied to piecewise deterministic Markov processes and transport equations with abstract boundary conditions. We give particular examples to illustrate our results.
\end{abstract}


\maketitle

\section{Introduction}

Stochastic models in natural sciences involving deterministic motion or growth and random jumps are particular examples of piecewise deterministic Markov processes (PDMPs) as introduced by Davis \cite{davis84}, see \cite{rudnickityran17}. These are processes whose sample paths (trajectories) are deterministic on random intervals $(\tau_n,\tau_{n+1})$, where $\tau_n$ is an increasing sequence of positive random variables, called jump times. The process is described with the help of three characteristics $(\phi,q,\jump)$ which are a flow $\phi=\{\phi_t\}_{t\in \mathbb{R}}$ determining deterministic paths of the process, a nonnegative jump rate function  $q$
and a transition  probability $\jump(x,B)$, specifying the distribution of jump from the point $x$ to a point in the set $B$. Let us consider the flow $\{\phi_t\}_{t\in \mathbb{R}}$ on $\mathbb{R}^N$ generated by a
globally Lipschitz continuous vector field $b\colon \mathbb{R}^N\to\mathbb{R}^N$,  so that for each $x_0\in \mathbb{R}^N$ the unique solution of the initial value problem
\begin{equation}\label{e:vjp}
x'(t)=b({x}(t)), \quad {x}(0)=x_0,
\end{equation}
is given by ${x}(t)=\phi_t(x_0)$. Given a set $E^0\subset \mathbb{R}^N$ we introduce the \emph{outgoing} boundary $\Gamma^{+}$  and the \emph{incoming} boundary $\Gamma^{-}$ which are points  of the boundary $\partial E^0$ of   $E^0$  through which the flow can leave the set $E^0$  and enter the set $E^0$, respectively. They are given by
\begin{equation*}
\Gamma^{\pm}=\{z\in \partial E^0\setminus E^0: z=\phi_{\pm t}( x)  \text{ for some }x\in E^0, t>0, \text{ and } \phi_{\pm s}(x)\in E^0, s\in [0,t)\}.
\end{equation*}
Starting at time $\tau_0=0$  from $X(\tau_0)=X_0=x_0$ with $x_0$ in the state space $E\subset E^0\cup\partial E^0$, the Markov process $X(t)$ follows the trajectory $\phi_t (x_0)$ until the first jump time $\tau_1$ that is defined by either reaching  the boundary $\Gamma^{+}$ or through a random disturbance occurring with intensity $q$ depending on the current position of the process. Then the value $X_1$ of the process at the jump time $\tau_1$ is selected according to $\Pr(X_1\in B|\phi_{\tau_1-\tau_0}(X_0)=x)=\jump (x, B)$ and the process restarts afresh from $X_1$. In this way we select a sequence of jump times $\tau_1<\tau_2<\tau_3 <\ldots$ and a sequence of post-jump values $X_1, X_2, X_3,\ldots$ allowing to define the paths of the process $X=\{X(t)\}_{t\ge 0}$ by
\begin{equation*}
X(t)=\left\{
  \begin{array}{ll}
  \phi_{t-\tau_{n-1}}(X_{n-1}) & \text{ for } \tau_{n-1}\le t<\tau_{n},\\
X_n &\text{ for } t=\tau_n;
\end{array}\right.
\end{equation*}
if  $\tau_\infty:=\lim_{n\to \infty}\tau_n<\infty$ we set $X(t)=\Delta$ for $t\ge \tau_\infty$, where $\Delta$ is a point at infinity. Therefore, the process is defined for all times and it will be called  \emph{the minimal process  with characteristics} $(\phi,q,\jump)$, see \cite{davis93,rudnickityran17} for details.
Let the state space be equipped with a  $\sigma$-finite measure $m$.
By  imposing general conditions on the characteristics (see Theorem~\ref{th:ptsmal}), we showed in \cite{GMTK} the existence of a substochastic semigroup (a positive contraction $C_0$-semigroup of linear operators) on $L^1(E,m)$ describing the evolution of densities for the process.
However, in general it might happen that $\tau_\infty$ is finite with positive probability, so that the minimal process is explosive,  leading to a loss of mass. So the question remains, when the process is non-explosive or, equivalently, the induced  semigroup is stochastic (each operator preserves the norm on positive cone).

A widely used class of mathematical models to describe spatial motion of individuals are velocity-jump processes in which individuals move in $\mathbb{R}^N$ with a constant velocity  and discontinuous changes in the speed or direction of an individual are generated at constant rate according to a Poisson process, see  \cite{Stroock,HH}  for more involved models. These are examples of PDMPs that can be also used  in the kinetic theory of gases or in  neutron transport to model transport of particles  (molecules of gas or neutrons).  Particles move in a bounded region and change randomly their velocities due to collisions with particles of the medium or by hitting the walls. These processes  are usually modeled with linear Boltzmann or linear transport equations  with boundary conditions describing interactions between the particles and the solid walls, see  \cite{beals87,greenberg87,cercignani88,mustapha97,villani02,cercignani02,banasiakarlotti06} and the references therein.
The state $x$ of a particle is described by a position $\mathrm{x}\in \Omega$ and a velocity $\mathrm{v}\in V$, where
$\Omega$ is a sufficiently smooth open subset of $\mathbb{R}^d$ and $V$ is a Borel subset of $\mathbb{R}^d$.
If an external force field $\mathrm{F}$  is present, then  the vector field $b$ on $\mathbb{R}^{2d}$ is of the form $b(\mathrm{x},\mathrm{v})=(\mathrm{v},\mathrm{F}(\mathrm{x},\mathrm{v}))$ as in the Vlasov  equation.
A particular example  is the free transport with $\mathrm{F}\equiv 0$ and the flow
\[
\phi_t(\x,\z)=(\x+ \z t,\z),\quad \x\in \Omega, \z\in V.
\]
We take $E^0=\Omega\times V$ and $m=\Leb\times \nu$, where $\Leb$ is the Lebesgue measure on $\mathbb{R}^d$ and $\nu$ is a Radon measure on $\mathbb{R}^d$ with support $V$.
We have
\[
 \Gamma^{\pm} =\{(\x,\z)\in \partial \Omega\times V: \pm \z\cdot n(\x) >0\},
\]
where $n(\x)$ is the outward normal at $\x\in \partial \Omega$.
We assume that a particle at position $\x\in \Omega$ and with velocity $\z\in V$ changes  its velocity with intensity $q(\x,\z)$ and  chooses a new  velocity according to  the following transition probability
\[
\jump((\x,\z),B)=\int_{V}\mathbf{1}_{B}(\x,\z')p(\x,\z,\z')\nu(d\z'),\quad B\in \B(\Omega\times V), (\x,\z)\in \Omega\times V,
\]
where $p$ is a measurable nonnegative function  defining  the scattering kernel $\kappa$ and satisfying
\begin{equation*}
\kappa(\x,\z',\z):=q(\x,\z)p(\x,\z,\z')
\quad\text{and}\quad  \int_{V} p(\x,\z,\z')\mu(d\z')=1,\quad \x\in \Omega,\z,\z'\in V.
\end{equation*}
To complete the description of the process one needs to define the jump distribution $\jump$ on $\Gamma^{+}$
satisfying $\jump((\x,\z),\Gamma^{-})=1$ for $(\x,\z)\in \Gamma^{+}$. Different types of boundary conditions were introduced in \cite{cercignani88}, see also \cite{lods05} and \cite{lods2018invariant}.
These are typically described by Maxwell-type boundary conditions stating that if a particle reaches the boundary $\Gamma^{+}$ at the point $(\x,\z)$ then with probability $\alpha(x)$ it undergoes a specular reflection and with probability $1-\alpha(x)$ it undergoes a diffuse-type reflection.

In
\cite{GMTK}
our  main tool was a perturbation result for substochastic semigroups on $L^1$ spaces. We considered
initial-boundary value problems  given in the general abstract form
\begin{equation}\label{e:eq}
u'(t) =Au(t)+Bu(t),\quad \Psi_0 u(t)=\Psi u(t),\quad t> 0,\quad u(0)=f,
\end{equation}
where $\Psi_0,\Psi$ are positive and possibly unbounded operators defined on a linear subspace $\mathcal{D}\subset L^1$ with values in a boundary space $L^1_{\partial}$, the operator $B\colon \mathcal{D}\to L^1$ is positive
and $A\colon \mathcal{D}\to L^1$ is such that
the operator $A_0$, defined as the restriction of $A$ to the nullspace $\Ker(\Psi_0)$, i.e.
\begin{equation}\label{d:dombou0}
A_0f=Af, \quad f\in\mathcal{D}(A_0)=\{f\in \mathcal{D}: \Psi_0f=0\}=\Ker(\Psi_0),
\end{equation}
is the generator of a substochastic semigroup on $L^1$. For example, if $m$ is the Lebesgue measure on $\mathbb{R}^N$ and $X$ is the minimal process with characteristics $(\phi,q,\jump)$  then the density $u(t)$ of $X(t)$ should satisfy \eqref{e:eq} with
the operators $A$ and $\Psi_0$ of the form
\[
Af(x)=- \nabla_{x} \cdot(b(x) f(x)) -q(x)f(x),\quad x\in E^0, \quad \Psi_0(f)(x)=f_{|\Gamma^{-}}(x), \quad x\in \Gamma^{-},
\]
for sufficiently smooth $f$ with $f_{|\Gamma^{-}}$ being the trace 
of $f$ on the incoming part $\Gamma^{-}$ of the boundary, while the operators $\Psi$ and $B$ are connected with jumps to the boundary $\Gamma^{-}$ and to the set $E^0$, respectively. Then it is relatively easy to show the well-posedness for problem \eqref{e:eq} with $B\equiv 0$ and $\Psi\equiv 0$ so that the assumption that the operator $(A_0,\mathcal{D}(A_0))$ as in \eqref{d:dombou0} is the generator of a substochastic semigroup will be satisfied. However, if $\Psi$ and $B$ are both unbounded the well-possedness of the problem \eqref{e:eq} might not hold in general.
In \cite{GMTK} we provided sufficient conditions  for the operator $A+B$ with domain $\Ker(\Psi-\Psi_0)$ to have an extension $G$ generating a substochastic semigroup on $L^1$.  In the particular case of  $\Psi_0=\Psi=0$ we recovered Kato's perturbation theorem \cite{kato54,voigt87,banasiakarlotti06} on $L^1$.
If $B\equiv0$ and $\Psi$ is bounded we have a particular example of a boundary perturbation as in Greiner’s perturbation theorem \cite{greiner}.

In Section \ref{s:2}, we extend
Kato's theorem \cite{kato54} to positive perturbations $\B$ of
operators $\A$ that act on abstract AL spaces $\BS$ and  are not densely defined, i.e.  $\overline{\mathcal{D}(\A)}\neq \BS$.
In  Theorem~\ref{th:main} we give sufficient conditions for the existence of  a substochastic semigroup on $\BS_0=\overline{\mathcal{D}(\A)}$ with  generator $(\G,\mathcal{D}(\G))$
being an extension of the part $(\A+\B)_{|}$ of the operator $\A+\B$ in $\BS_0$, i.e.
\begin{equation*}
\G u =\A u+\B u \quad \text{for } u\in \mathcal{D}((\A+\B)_{|})=\{u\in \mathcal{D}(\A): \A u+\B u \in \BS_0\}
\end{equation*}
and $\mathcal{D}((\A+\B)_{|})\subseteq \mathcal{D}(\G)$.
We also provide
 necessary and sufficient conditions for $\G$ to be equal to the operator $(\A+\B)_{|}$ or to its closure $\overline{(\A+\B)_{|}}$.
Going back to equation \eqref{e:eq}  the space $\BS$ is taken to be $L^1\times L^1_{\partial}$ 
and  the operators $\mathcal{A},\mathcal{B}\colon \mathcal{D}(\mathcal{A}) \to L^1\times L^1_{\partial}$ with $\mathcal{D}(\mathcal{A})=\mathcal{D} \times \{0\}$ are defined by
\begin{equation}\label{d:AB0}
\mathcal{A}(f,0)=(Af,-\Psi_{0}f)\quad \text{and}\quad \mathcal{B}(f,0)=(B f,\Psi f) \quad
\textrm{for $f\in \mathcal{D}$}.
\end{equation}

In Section \ref{s:3}, we show how the results from Section~\ref{s:2} 
can be applied to problems as in \eqref{e:eq} with both $B=0$ and $B\neq 0$.
We also complete the characterization of the generator $G$ from \cite{GMTK}.

Finally, Section~\ref{s:4} contains applications of our abstract results to PDMPs and to transport 
equations
with conservative boundary conditions, where we show that the semigroup is stochastic if and only if the generator $G$ is the closure of the corresponding operator $A_{\Psi}+B$.   We revisit and generalize results from~\cite{arlottibanasiaklods07,arlottibanasiaklods11,arlottilods14}. Here the boundary conditions
are given in
an abstract form
\begin{equation*}
f_{|\Gamma^{-}}=H(  f_{|\Gamma^{+}})
,
\end{equation*}
where  $f_{|\Gamma^{-}}$ is the trace %
of $f$ on the incoming part $\Gamma^{-}$ of the boundary,
$f_{|\Gamma^{+}}$ is the trace
of $f$ on the outgoing part $\Gamma^{+}$ of the boundary and $H$ is a positive operator acting between the trace spaces $ L^1(\Gamma^{+},m^{+})$ and $ L^1(\Gamma^{-}, m^{-})$, where $m^{\pm}$ are suitable Borel measures on $\Gamma^{\pm}$.
A particular attention attracted the well-posedness of the collisionless kinetic equation, i.e.  $B\equiv 0$ and $q\equiv 0$,  see \cite{voigt80,greenberg87,mustapha99,lods05,arlottilods05,mustapha08,arlottibanasiaklods11,arlottilods14} and the references therein.
The case of unbounded $B$ and $q$ with arbitrary vector fields with no-reentry boundary conditions (i.e. $H=0$) was studied in \cite{arlottibanasiaklods07}. Non-zero $H$ was mainly treated by first showing the well-posedness for the collisionless equation with the given boundary condition   and then applying perturbation arguments to get existence of solutions for the full equation, see \cite{arlotti91,mee01,banasiakarlotti06}. In the case of dissipative $H$, i.e. with norm strictly less than 1 the collisionless operator generates a substochastic semigroup, while in 
the case where the norm of the boundary operator $H$ is one can be still dealt with but in general only the existence can be shown without uniqueness. The last section contains a pair of examples.

\section{Positive perturbations of non-densely defined operators on abstract AL spaces}
\label{s:2}

In this section we assume that $\BS$ is an abstract AL space, i.e. a Banach lattice where the norm is additive on the positive cone, $\|u+v\|=\|u\|+\|v\|$ for $u,v\in \BS_+$. There exists a unique positive functional $\alpha\colon \BS\to \mathbb{R}$
such that
\begin{equation}\label{d:alp}
\alpha(u)=\langle \alpha,u\rangle=\|u\|,\quad u\in \BS_+.
\end{equation}

Let $(\A,\mathcal{D}(\A))$ be a linear operator on $\BS$.  It is said to be \emph{positive} if $\A u\in \BS_+$ for $u\in \mathcal{D}(\A)\cap \BS_+$. We write $\mathcal{I}$ for the identity operator on  $\BS$.
We recall that if for some real $\lambda$ the operator
$\lambda-\A:=\lambda \mathcal{I}-\A$ is bijective and $(\lambda -\A)^{-1}$ is a bounded linear operator, then $\lambda$ is said to belong to the \emph{resolvent set} $\rho(\A)$ and $R(\lambda,\A):=(\lambda  -\A)^{-1}$ is called the \emph{resolvent operator}
 of $\A$ at $\lambda$.
Following  \cite{arendt87},
a linear operator $\A$ is said to be \emph{resolvent positive} if there exists
$\omega\in\mathbb{R}$ such that $(\omega,\infty)\subseteq\rho(\A)$ and the resolvent operator
$R(\lambda,\A):=(\lambda-\A)^{-1}$ is positive for all $\lambda>\omega$. Generators of substochastic semigroups are resolvent positive. A family $\{\mathcal{S}(t)\}_{t\ge 0}$ of bounded linear operators on a (given) closed subspace $\BS_0$ of $\BS$ is called a \emph{substochastic (stochastic) semigroup}  on $\BS_0$ if it is a $C_0$-semigroup and each operator $\mathcal{S}(t)\colon \BS_0\to \BS_0$ is \emph{substochastic} (resp. \emph{stochastic}), i.e. $\mathcal{S}(t)$ is positive   and $\|\mathcal{S}(t)u\|\le \|u\|$ (resp. $\|\mathcal{S}(t)u\|= \|u\|$) for $u\in \BS_0\cap\BS_+$. 

Suppose now that
$(\A,\mathcal{D}(\A))$ is resolvent positive and such that
\begin{equation}\label{e:subsem}
\alpha(\A u) \le 0\quad \text{for all  }u\in \mathcal{D}(\A)_+:=\mathcal{D}(\A)\cap \BS_+.
\end{equation}
Then $(0,\infty)\subseteq \rho(\A)$ and  $\lambda \|R(\lambda,\A)\|\le 1$ for all $\lambda>0$. If, additionally,  $\mathcal{D}(\A)$ is dense in $\BS$, then the operator $(\A,\mathcal{D}(\A))$ is the generator of a substochastic semigroup on $\BS$, by the Hille-Yosida theorem \cite{engelnagel00}. It is easy to see that condition \eqref{e:subsem} is also necessary for $\A$ to be the generator of a substochastic semigroup.
Moreover,  equality holds in \eqref{e:subsem} if and only if  $(\A,\mathcal{D}(\A))$  generates a  stochastic semigroup.  If the operator $\A$ is not densely defined
then
the part of $\A$ in $\BS_0=\overline{\mathcal{D}(\A)}$ is densely defined in $\BS_0$ and generates a substochastic semigroup on $\BS_0$, see, for example, Corollary II.3.21 in~\cite{engelnagel00}.
We recall that \emph{the part} of $\A$ in $\BS_0=\overline{\mathcal{D}(\A)}$, denoted by $\A_{|}$, is the restriction of $\A$ to the domain
\[
\mathcal{D}(\A_|)=\{u\in \mathcal{D}(\A)\cap \BS_0: \A u\in \BS_0\}.
\]

Arguing as in \cite[Theorem 3.1]{GMTK} (see also \cite[Theorem 2.1]{arlottilodsmokhtar11})  we prove the following perturbation result
extending Kato's perturbation theorem  \cite{kato54,voigt87,thiemevoigt06,banasiakarlotti06} to operators with non-dense domains on abstract AL spaces.

\begin{theore}\label{th:main}
Let $(\A,\mathcal{D}(\A))$ be a resolvent positive operator on $\BS$ and $\B\colon\mathcal{D}(\A)\to \BS$ be a positive operator such that \begin{equation}\label{d:alpha}
\alpha (\A u+\B u) \le 0\quad \text{for all }u\in \mathcal{D}(\A)_+.
\end{equation}
Then there exists a  substochastic semigroup $\{\mathcal{P}(t)\}_{t\ge 0}$ on $\BS_0=\overline{\mathcal{D}(\A)}$ with generator $(\G,\mathcal{D}(\G))$
being an extension of the part $(\A+\B)_{|}$ of $(\A+\B,\mathcal{D}(\A))$ in $\BS_0$, i.e.
\begin{equation}
\G u =\A u+\B u \quad \text{for } u\in \mathcal{D}((\A+\B)_{|})=\{u\in \mathcal{D}(\A): \A u+\B u \in \BS_0\}
\end{equation}
and $\mathcal{D}((\A+\B)_{|})\subseteq \mathcal{D}(\G)$.
The resolvent operator of $\G$ at $\lambda>0$ is given by
\begin{equation}\label{d:resgG}
R(\lambda,\G)u=\lim_{N\to \infty} \sum_{n=0}^{N}R(\lambda,\A)(\B R(\lambda,\A))^n u,\quad u\in \BS_0.
\end{equation}
\end{theore}

\begin{remar} In the context of Theorem~\ref{th:main} condition \eqref{d:alpha}  implies  \eqref{e:subsem}, since the operator $\mathcal{B}$ is positive. Hence  $(0,\infty)\subseteq \rho(\A)$ and the operator $R(\lambda,\A)$ is positive for all $\lambda>0$. Thus $\mathcal{B}R(\lambda,\mathcal{A})$ is also positive and condition \eqref{d:alpha} implies that  $\mathcal{B}R(\lambda,\mathcal{A})$ is  a substochastic operator on~$\BS$ for $\lambda>0$.
Note also that we have
\begin{equation}\label{e:main}
\lambda\mathcal{I} -(\mathcal{A}+\mathcal{B})=(\mathcal{I}-\mathcal{B}R(\lambda,\mathcal{A}))(\lambda \mathcal{I}- \mathcal{A}),
\end{equation}
where $\mathcal{I}$ is the identity operator on  $\BS$.  In particular, if the operator $\mathcal{I}-\mathcal{B}R(\lambda,\mathcal{A})$ is invertible with positive inverse, then
the resolvent of $\mathcal{A}+\mathcal{B}$ at $\lambda>0$ is given by
\begin{equation}\label{e:rab}
R(\lambda,\mathcal{A}+\mathcal{B})=R(\lambda,\mathcal{A})(\mathcal{I}-\mathcal{B}R(\lambda,\mathcal{A}))^{-1}
\end{equation}
and $\lambda \|R(\lambda,\mathcal{A}+\mathcal{B})\|\le 1$.
\end{remar}

\begin{proof}
For each $r\in [0,1)$ we define  the operator
\[
\G_r=\mathcal{A}+r\mathcal{B},\quad \mathcal{D}(\G_r)=\mathcal{D}(\A).
\]
Since $\|r \B R(\lambda,\A)\|\le r<1$, we obtain
\[
R(\lambda,\mathcal{G}_r)=R(\lambda,\mathcal{A})(\mathcal{I}-r \B R(\lambda,\A))^{-1}=R(\lambda,\mathcal{A})\sum_{n=0}^\infty r^n (\mathcal{B}R(\lambda,\mathcal{A}))^n.
\]
We have $0\le R(\lambda,\mathcal{G}_r)\le R(\lambda,\mathcal{G}_{r'})$ for $r<r'$ and $\|R(\lambda,\mathcal{G}_r)\|\le \lambda^{-1}$. Thus the limit
\[
\mathcal{R}_\lambda u =\lim_{r\uparrow 1}R(\lambda,\mathcal{G}_r) u
\]
exists for all  $u\in \BS_+$ and
\begin{equation}\label{e:res}
\mathcal{R}_\lambda u=\lim_{r\uparrow 1}R(\lambda,\mathcal{G}_r)u=\lim_{N\to\infty}R(\lambda,\mathcal{A})\sum_{n=0}^N (\mathcal{B}R(\lambda,\mathcal{A}))^n u,\quad u\in \BS.
\end{equation}

The part ${\G_r}_{|}$ of $\G_r$
in $\BS_0=\overline{\mathcal{D}(\A)}$ is the generator of a substochastic semigroup $\{\mathcal{P}_{r}(t)\}_{t\ge 0}$  on $\BS_0$.
The substochastic semigroup $\{\mathcal{P}(t)\}_{t\ge 0}$ is defined by
\begin{equation}\label{d:Plim}
\mathcal{P}(t)u=\lim_{r\to 1}\mathcal{P}_{r}(t)u,\quad u\in \BS_0;
\end{equation}
the convergence is uniform for $t$ in compact subsets of $[0,\infty)$.
If  $\G$ is its generator then
$R(\lambda,\mathcal{G})$ is the part ${\mathcal{R}_{\lambda}}_{|}$ of the operator $\mathcal{R}_\lambda$ in $\BS_0$, where $\mathcal{R}_\lambda$ is defined by \eqref{e:res}.
Since
\[
R(\lambda,\mathcal{A})\sum_{n=0}^N (\mathcal{B}R(\lambda,\mathcal{A}))^n(\lambda \mathcal{I}-\mathcal{A})u=u+R(\lambda,\mathcal{A})\sum_{n=0}^{N-1} (\mathcal{B}R(\lambda,\mathcal{A}))^n\mathcal{B}u
\]
for all $N$ and $u\in \mathcal{D}(\A)$, we see that
\begin{equation}\label{e:inje}
\mathcal{R}_\lambda (\lambda \mathcal{I}-\mathcal{A}-\mathcal{B})u=u
\end{equation}
for $u\in \mathcal{D}(\A)$, by \eqref{e:res}. Now if $u\in \mathcal{D}((\A+\B)_{|})$ then $(\lambda \mathcal{I}-\mathcal{A}-\mathcal{B})u\in \BS_0$, implying that $\mathcal{G}$ is an extension of the operator $(\mathcal{A}+\mathcal{B})_{|}$.
\end{proof}

\begin{remar}\label{r:ALM} Theorem \ref{th:main} remains valid when $\BS$ is as in \cite{arlottilodsmokhtar11}, i.e., a real ordered Banach space with generating normal positive cone $\BS_+$ and additive norm on $\BS_+$. Then we need to assume that the subspace $\BS_0$ has the same properties as $\BS$. 
\end{remar}

\begin{remar}\label{r:kernel} Note that in the setting of Theorem \ref{th:main}
the operators $\mathcal{I}-\B R(\lambda,\A)$ and $\lambda \mathcal{I}-\A-\B$ are injective  for all $\lambda>0$, by~\eqref{e:main} and~\eqref{e:inje}.
\end{remar}

We have the following characterization extending the results of \cite{frosali04}, see also the spectral criterion in  \cite[Theorem 4.3]{banasiakarlotti06}.

\begin{theore}\label{th:spp}
Under the assumptions of Theorem \ref{th:main} the following hold:
\begin{enumerate}[\upshape(i)]
\item\label{i:1} The generator $(\G,\mathcal{D}(\G))$ is the operator $((\A+\B)_{|},\mathcal{D}((\A+\B)_{|}))$ if and only if the range of the operator $\mathcal{I}-\B R(\lambda,\A)\colon \BS\to \BS$ contains $\BS_0$ for some/all $\lambda>0$.
\item\label{i:2} The generator $(\G,\mathcal{D}(\G))$ is the closure of  $((\A+\B)_{|},\mathcal{D}((\A+\B)_{|}))$ if and only if the closure of the range of $\mathcal{I}-\B R(\lambda,\A)$ contains $\BS_0$ for some/all $\lambda>0$.
\end{enumerate}
\end{theore}

\begin{proof}
To prove \eqref{i:1} first, note that
\begin{equation*}\label{e:closure0}
\mathcal{D}((\A+\B)_{|})=R(\lambda,\G)(\BS_0\cap(\mathcal{I}-\B R(\lambda,\A))(\BS))
\end{equation*}
for all $\lambda>0$. We have $\mathcal{D}(\G)=R(\lambda,\G)(\BS_0)$. This implies that $\mathcal{D}(\G)=\mathcal{D}((\A+\B)_{|})$ if and only if $\BS_0=\BS_0\cap (\mathcal{I}-\B R(\lambda,\A))(\BS)$ and point \eqref{i:1} follows.

Now, we show that
\begin{equation}\label{e:closure}
\mathcal{D}(\overline{(\A+\B)_{|}})=R(\lambda,\G)(\overline{\BS_0\cap(\mathcal{I}-\B R(\lambda,\A))(\BS)})
\end{equation}
for all $\lambda>0$ (the proof was kindly communicated by one of the referees).  First suppose that $u\in \mathcal{D}(\overline{(\A+\B)_{|}})$ and put $f=(\lambda I-\overline{(\A+\B)_{|}})u$. Then there exists a sequence $u_n$ of elements from  $\mathcal{D}((\A+\B)_{|})$ such that $u_n\to u$ and $f_n:=(\lambda I-\overline{(\A+\B)_{|}})u_n\to f\in \BS_0$. Since $u_n\in \mathcal{D}(\A)$ and $(\A+\B)u_n\in \BS_0$, we see that $f_n\in \BS_0$ for all $n$. For each $n$ we can find $g_n\in \BS$ such that $u_n=R(\lambda,\A)g_n$. Thus $f_n=(\mathcal{I}-\B R(\lambda,\A))g_n\in (\mathcal{I}-\B R(\lambda,\A))(\BS)$ showing that $f_n\in \BS_0\cap (\mathcal{I}-\B R(\lambda,\A))(\BS)$ and implying that $f\in \overline{\BS_0\cap(\mathcal{I}-\B R(\lambda,\A))(\BS)}$. Consequently, $u=R(\lambda,\G)f\in R(\lambda,\G)(\overline{\BS_0\cap(\mathcal{I}-\B R(\lambda,\A))(\BS)})$.
Conversely, we take $u\in R(\lambda,\G)(\overline{\BS_0\cap(\mathcal{I}-\B R(\lambda,\A))(\BS)})$ and we define $f=(\lambda \mathcal{I}-\G)u$. Then there exists a sequence $g_n\in \BS$  such that $f_n=(\mathcal{I}-\B R(\lambda,\A))g_n\in \BS_0$ and $f_n\to f$ as $n\to \infty$. We have $u_n:=R(\lambda,\A)g_n\in \mathcal{D}(\A)$ and $f_n=(\lambda\mathcal{I}-(\A+\B))u_n$ showing that $u_n\in \mathcal{D}((\A+\B)_{|})$ for any $n$. Hence, $u_n=R(\lambda,\G)f_n\to u$ as $n\to \infty$ and $u\in \mathcal{D}(\overline{(\A+\B)_{|}})$.

Relation \eqref{e:closure} implies that $\G=\overline{(\A+\B)_{|}}$ if and only if $\BS_0=\overline{\BS_0\cap(\mathcal{I}-\B R(\lambda,\A))(\BS)}$, completing the proof of \eqref{i:2}.
\end{proof}

Remark~\ref{r:kernel} and  Theorem~\ref{th:spp} have the following immediate consequence, where  $\sigma_p$ and $\sigma_c$ denote the point and  the continuous spectrum, respectively.
\begin{corollar}\label{c:suffi}
We have the following
\begin{enumerate}[\upshape(i)]
\item $1\not\in \sigma_p(\B R(\lambda,\A))$ for all $\lambda>0$.
\item If $1\in \rho(\B R(\lambda,\A))$ for some $\lambda>0$ then $\G=(\A+\B)_{|}$.
\item If $1\in \sigma_c(\B R(\lambda,\A))$ for some $\lambda>0$ then $\G=\overline{(\A+\B)_{|}}$.
\end{enumerate}
\end{corollar}

\begin{remar}\label{r:spectral}
Note that if $\mathcal{T}$ is a substochastic operator such that $\Ker(\mathcal{I}-\mathcal{T})=\{0\}$, then the following are equivalent
\begin{enumerate}[(i)]
\item $1\in \rho(\mathcal{T})$,
\item $r(\mathcal{T})<1$, where
$r(\mathcal{T})$
denotes the spectral radius of $\mathcal{T}$, i.e.,
$r(\mathcal{T})=\lim_{n\to\infty}\sqrt[n]{\|\mathcal{T}^n\|}$,
\item $\lim_{n\to \infty}\|\mathcal{T}^n\|=0$,
\item the operator $\mathcal{T}$ is quasi-compact, i.e., there exist a compact
operator $\mathcal{K}$ and $n\in\mathbb{N}$ such that $\|\mathcal{T}^n-\mathcal{K}\|<1$.
\end{enumerate}
\end{remar}

We next prove the following fundamental result, extending the results of \cite{kato54,voigt87,banasiakarlotti06}.

\begin{theore}\label{th:equiv} Let $\lambda>0$.
Under the assumptions of Theorem \ref{th:main}
the following conditions are equivalent:
\begin{enumerate}[\upshape(i)]
\item\label{i:1t} The generator $(\G,\mathcal{D}(\G))$ of the substochastic semigroup $\{P(t)\}_{t\ge 0}$ is the closure of the operator $((\A+\B)_{|},\mathcal{(\A+\B)_{|}})$.
\item\label{i:2t} We have
\begin{equation}\label{e:merg}
\lim_{n\to\infty} (\B R(\lambda,\A))^n u=0,\quad u\in \BS_0.
\end{equation}
\end{enumerate}
If, additionally,
\begin{equation}\label{d:alphaz}
\alpha (\A u+\B u)= 0\quad \text{for all }u\in \mathcal{D}(\A)_+,
\end{equation}
then \eqref{i:1t} and \eqref{i:2t}  are also equivalent to:
\begin{enumerate}[\upshape(i)]\setcounter{enumi}{2}
\item\label{i:3t}  The semigroup $\{P(t)\}_{t\ge 0}$  is stochastic.
\end{enumerate}
\end{theore}
\begin{proof}
To prove the equivalence of \eqref{i:1t} and \eqref{i:2t} we make use of Theorem~\ref{th:spp} part \eqref{i:2} and show that condition \eqref{e:merg} is  equivalent to $\BS_0\subseteq \overline{(\mathcal{I}-\B R(\lambda,\A))(\BS)}$.
Since the operator $\B R(\lambda,\A)$ is substochastic,
we obtain
\begin{equation}\label{e:meanerg}
\overline{(\mathcal{I}-\B R(\lambda,\A))(\BS)}=\{u\in \BS: \lim_{N\to \infty}\frac{1}{N}\sum_{n=0}^{N-1}(\B R(\lambda,\A))^n u=0\},
\end{equation}
by the Yosida theorem \cite[Theorem 2.1.3]{krengel85}. Now, if \eqref{e:merg} holds then it follows from \eqref{e:meanerg} that
$\BS_0\subseteq \overline{(\mathcal{I}-\B R(\lambda,\A))(\BS)}$. Conversely, take any  $u\in \BS_0\subseteq \overline{(\mathcal{I}-\B R(\lambda,\A))(\BS)}$. We have $|u|\in \BS_0$
and the sequence $\|(\B R(\lambda,\A))^n |u|\|$ is convergent. We get
\[
\lim_{n\to \infty}\|(\B R(\lambda,\A))^n |u|\|=\lim_{N\to \infty}\|\frac{1}{N}\sum_{n=0}^{N-1}(\B R(\lambda,\A))^n |u|\|
  =0,
\]
by additivity of the norm and \eqref{e:meanerg}. This completes the proof of the first equivalence, since $ \|(\B R(\lambda,\A))^n u\|\le \|(\B R(\lambda,\A))^n |u|\|$ for any $u$ and $n$.

Now suppose that condition \eqref{d:alphaz} holds.
Note that a substochastic semigroup with generator $\G$ is stochastic if and only if there is $\omega\in \mathbb{R}$ such that
the operator $\lambda R(\lambda,\G)$ is stochastic for  all $\lambda>\omega$. Since $\mathcal{B}R(\lambda,\mathcal{A})$ is a substochastic operator, condition \eqref{e:merg} holds for all sufficiently large $\lambda>0$ if it holds for one $\lambda>0$. Thus $\G$ is the generator of a  stochastic semigroup if and only if the operator $\lambda R(\lambda,\G)$ is stochastic for all $\lambda$ satisfying \eqref{e:merg}.
Observe that \eqref{d:alphaz} together with \eqref{e:main} leads to
\begin{equation}\label{c:stochres}
\|\lambda R(\lambda,\mathcal{A})u\|=\|u\|-\|\mathcal{B}R(\lambda,\mathcal{A})u\|
\end{equation}
for all  $u\in \BS_+$. Hence, for  $u\in \BS_0\cap \BS_+$ and for each $N$ we obtain
\[
\lambda \|R(\lambda, \A)\sum_{n=0}^N (\mathcal{B}R(\lambda,\mathcal{A}))^n u\|=\|u\|-\|(\mathcal{B}R(\lambda,\mathcal{A}))^{N+1}u\|
\]
By taking the limit as $N\to \infty$, we see that
\[
\lambda\| R(\lambda,\G)u\| =\|u\|-\lim_{N\to\infty} \|(\mathcal{B}R(\lambda,\mathcal{A}))^{N}u\|,
\]
by \eqref{d:resgG} and the equivalence follows.
\end{proof}

We can also extend \cite[Theorem 3.6]{tyran09} to the situation studied in this paper.

\begin{theore}\label{t:opK}
Under the assumptions of Theorem~\ref{th:main} the operator $K\colon \BS\to \BS$ defined by
\begin{equation}\label{eq:K}
K u=\lim_{\lambda\downarrow 0}\mathcal{B}R(\lambda,\mathcal{A})u.
\end{equation}
is substochastic.  Moreover, if $K$ is mean ergodic on $\BS_0$, i.e.,
\begin{equation}\label{d:meaner}
\lim_{N\to\infty}\frac{1}{N}\sum_{n=0}^{N-1}K^n u \quad \text{exists for all } u\in\BS_0,
\end{equation}
then \eqref{e:merg} holds.
\end{theore}
\begin{proof} The proof of the first part  is as in  \cite{tyran09}. Note that  $\mathcal{B}R(\mu,\mathcal{A})\le \mathcal{B}R(\lambda,\mathcal{A})\le K$ for $0<\lambda<\mu$.  For any nonnegative $u\in \BS_0$ we have
\[
0\le \frac{1}{N}\sum_{n=0}^{N-1}(\mathcal{B}(R(\lambda,\mathcal{A})))^n u\le \frac{1}{N}\sum_{n=0}^{N-1}K^n u
\]
and  \eqref{d:meaner} holds. Thus the set $\{\frac{1}{N}\sum_{n=0}^{N-1}(\mathcal{B}(R(\lambda,\mathcal{A})))^n u:N\ge 1\}$ is conditionally weakly compact.
 By Remark \ref{r:kernel} we have $\Ker(\mathcal{I}-\mathcal{B}R(\lambda,\mathcal{A}))=\{0\}$ for $\lambda>0$. This together with the mean ergodic theorem \cite[Theorem 2.1.1]{krengel85} gives
\[
\lim_{N\to \infty}\frac{1}{N}\sum_{n=0}^{N-1}(\mathcal{B}(R(\lambda,\mathcal{A})))^n u=0
\]
for all nonnegative $u\in \BS_0$. Additivity of the norm implies now that \eqref{e:merg} holds and completes the proof.
\end{proof}

\begin{remar}\label{r:meanerg}
Note that  a substochastic operator $K$ on an $L^1$ space is mean ergodic on $L^1$ if and only if it is weakly almost periodic, i.e. the set $\{K^n u:n\ge 0\}$ is relatively weakly compact for each $u\in L^1$,  see \cite{kornfeldlin00}. In particular, if  $Ku\le u$ for some $u\in L^1$ and $u$ is a quasi-interior element (i.e. $u>0$ a.e.),  then the operator $K$ is mean ergodic on $L^1$.
\end{remar}

\begin{remar}
Theorem~\ref{th:spp},  Corollary \ref{c:suffi} and Theorem~\ref{t:opK}  are valid when   $\BS$ is as in \cite{arlottilodsmokhtar11}, see Remark~\ref{r:ALM}. In Theorem~\ref{th:equiv} we used the lattice property to prove that condition \eqref{i:1t} implies \eqref{e:merg} while the proof of the converse implication is valid in general Banach spaces.
\end{remar}

\section{Perturbations of boundary conditions}
\label{s:3}

In this section we revisit the perturbation theorem for substochastic semigroups from \cite{GMTK}
and show how the results from Section~\ref{s:2} can be used to obtain the characterization of the generator for problems as in \eqref{e:eq}.

Let $(E,\mathcal{E},m)$ and $(E_{\partial},\mathcal{E}_{\partial},m_{\partial})$ be two $\sigma$-finite measure spaces. Denote by $L^1=L^1(E,\mathcal{E},m)$ and $L^1_{\partial}=L^1(E_{\partial},\mathcal{E}_{\partial},m_{\partial})$ the corresponding spaces of integrable functions.
We consider linear operators $A\colon \mathcal{D}\to L^1$ and $\Psi_0\colon \mathcal{D}\to L^1_{\partial}$, where $\mathcal{D}$ is a linear subspace of $L^1$.
Our fundamental assumption is the following:
\begin{assumptio}\label{a:S2}
There exists $\omega\in \mathbb{R}$  such that for each $\lambda>\omega$ the operator $\Psi_0$ restricted to the nullspace $\Ker(\lambda -A)=\{f\in \mathcal{D}:\lambda f-Af=0\}$ has a positive right inverse, i.e. there exits a positive operator $\Psi(\lambda)\colon L^1_{\partial} \to \Ker(\lambda-A)$ such that $\Psi_0\Psi(\lambda)f_{\partial}=f_{\partial}$ for $f_{\partial}\in L^1_{\partial}$.
\end{assumptio}
We consider $\BS=L^1\times L^1_{\partial}$  with norm
\[
\|(f,f_{\partial})\|=
\int_E |f(x)|m(dx)+\int_{E_{\partial}}|f_{\partial}(x)|m_{\partial}(dx),\quad (f,f_{\partial})\in L^1\times L^1_{\partial},
\]
and we define the operator $\A\colon \mathcal{D}(\mathcal{A})\to \BS$ with $\mathcal{D}(\mathcal{A})=\mathcal{D} \times \{0\}$  by
\begin{equation}\label{d:AB}
\mathcal{A}(f,0)=(Af,-\Psi_{0}f)\quad
\textrm{for $f\in \mathcal{D}$}.
\end{equation}

We start with the following result.
\begin{lemm} Let the operators $A$ and $\Psi_0$ satisfy Assumption \ref{a:S2}.
Suppose that the operator $(A,\Ker(\Psi_0))$ is resolvent positive  and that
\begin{equation}\label{e:nuzero}
\int_E Af\, dm -\int_{E_{\partial}} \Psi_0 f\, dm_{\partial}\le 0 \quad \text{for all nonnegative } f\in \mathcal{D}.
\end{equation} Then the operator $\A$ defined in \eqref{d:AB}
is resolvent positive
 with the resolvent operator at $\lambda>0$ given by
\begin{equation}\label{eq:RA}
R(\lambda,  \mathcal{A})(f,f_{\partial})=(R(\lambda,A_0)f+\Psi(\lambda) f_{\partial},0),\quad (f,f_{\partial})\in L^1\times L^1_{\partial}, \lambda>0,
\end{equation}
and the part $\A_{|}$ of $\A$ in $\BS_0=\overline{\mathcal{D}(\A)}=\overline{\mathcal{D}}\times \{0\}$ is the generator of a substochastic semigroup on $\BS_0$. Moreover, $\overline{\mathcal{D}}=L^1$ if and only if $\overline{\Ker(\Psi_0)}=L^1$.
\end{lemm}
\begin{proof} By e.g \cite[Section 3.3.4]{rudnickityran17} we have \eqref{eq:RA}.
It follows from \eqref{e:nuzero} that the operator $(\A,\mathcal{D}(\A))$ satisfies \eqref{e:subsem}. Thus $(0,\infty)\subset \rho(\A)$, the part $\A_{|}$ of $\A$ in $\BS_0$ is the generator of a substochastic semigroup on $\BS_0$ and $\mathcal{D}(\A_{|})$ is dense in $\BS_0$.  We have $\Ker(\Psi_0)\subset \mathcal{D}$. Hence, if $\Ker(\Psi_0)$ is dense in $L^1$ then so is $\mathcal{D}$. Since
\[
\mathcal{D}(\A_{|})=\{(f,0)\in \mathcal{D}\times \{0\}: Af\in \overline{\mathcal{D}}, \Psi_0f=0\},
\]
if conversely $\mathcal{D}$ is dense in $L^1$ then we have $\mathcal{D}(\A_{|})=\Ker(\Psi_0)\times\{0\}$, thus $\Ker(\Psi_0)$ is also dense in $L^1$.
\end{proof}

We first consider problem  \eqref{e:eq} with $B=0$.
Given a positive operator $\Psi\colon \mathcal{D}\to L^1_\partial$  we define the operator  $(A_{\Psi},\mathcal{D}(A_\Psi))$  by
\begin{equation}
\label{d:dombou}
A_{\Psi}f=Af, \quad f\in \mathcal{D}(A_{\Psi})=\{f\in \mathcal{D}: \Psi_0f=\Psi f\}.
\end{equation}
Thus, starting with the operator $(A_0,\mathcal{D}(A_0))=(A,\Ker(\Psi_0))$ as in \eqref{d:dombou0} we perturb its domain and ask when this operator is again the generator of a substochastic semigroup.
We have the following generation result.

\begin{theore}\label{th:pertboun}
Let the operators $A$ and $\Psi_0$ satisfy Assumption \ref{a:S2} and let $(A_0,\mathcal{D}(A_0))=(A,\Ker(\Psi_0))$ be the generator of a substochastic semigroup on $L^1$.
Suppose that $\Psi\colon \mathcal{D}\to L^1_\partial$ is a positive operator such that
\begin{equation}\label{eq:zeroA}
\int_{E} Af \, dm+\int_{E_{\partial}}(\Psi f-\Psi_0f)\, dm_{\partial}\le 0\quad \text{for all nonnegative } f\in \mathcal{D}.
\end{equation}
Then there exists an extension $G_{\Psi}$ of the operator $A_{\Psi}$ defined in \eqref{d:dombou} generating a substochastic semigroup on $L^1$. We have
\begin{equation}\label{d:GpsiA}
\mathcal{D}(A_\Psi)\subseteq \mathcal{D}(G_{\Psi})\subseteq \mathcal{D},\quad G_{\Psi} f= Af \quad \text{for } f\in \mathcal{D}(G_{\Psi}),
\end{equation}
and the resolvent operator of $G_\Psi$ at $\lambda>0$ is given by
\begin{equation}\label{e:GPsi}
R(\lambda, G_{\Psi})f=R(\lambda,A_0)f+\sum_{n=0}^\infty \Psi(\lambda)(\Psi \Psi(\lambda))^n \Psi R(\lambda,A_0)f,\quad f\in L^1.
\end{equation}
Moreover,
\begin{enumerate}[\upshape(i)]
\item\label{i:ec1} $G_{\Psi}=A_{\Psi}$ if and only if  $\Psi (\Ker(\Psi_0))\subseteq (I_{\partial}-\Psi\Psi(\lambda))(L^1_{\partial})$ for some/all $\lambda>0$, where $I_{\partial}$ is the identity operator on $L^1_{\partial}$.
\item\label{i:ec2} $G_{\Psi}=\overline{A_{\Psi}}$ if and only $(\Psi\Psi(\lambda))^nf_{\partial}\to 0$ as $n\to \infty$ for all $f_{\partial}\in \Psi (\Ker(\Psi_0))$ and for some/all $\lambda>0$.
\end{enumerate}
\end{theore}
\begin{remar}
Note that $\Ker(\Psi_0)=R(\lambda,A_0)(L^1)$ for all $\lambda>0$ and $\Ker(I_{\partial}-\Psi\Psi(\lambda))=\{0\}$ for all $\lambda>0$. Thus if  $\Psi (\Ker(\Psi_0))\subseteq (I_{\partial}-\Psi\Psi(\lambda))(L^1_{\partial})$  then $(I_{\partial}-\Psi\Psi(\lambda))^{-1}$ is well defined on $\Psi(R(\lambda,A_0)(L^1))$ and
\begin{equation}\label{eq:resbou}
 R(\lambda,A_{\Psi})f=(I+\Psi(\lambda)(I_{\partial }-\Psi \Psi(\lambda))^{-1}\Psi)R(\lambda,A_0)f,\quad f\in L^1,\lambda>0.
\end{equation}
\end{remar}
\begin{proof}
On $\BS=L^1\times L^1_{\partial}$ we take the operator $\A$ as in \eqref{d:AB} and we define the operator $\B\colon \mathcal{D}(\A)\to \BS$  by
\begin{equation}\label{e:opB1}
\B(f,0)=(0,\Psi f),\quad f\in \mathcal{D}.
\end{equation}
Since \eqref{eq:zeroA} implies \eqref{e:nuzero}, we see that all assumptions of Theorem~\ref{th:main} hold. Observe that the part of $(\A+\B)_{|}$ in $\BS_0=L^1\times \{0\}$ is given by
\[
(\mathcal{A}+\mathcal{B})_{|}(f,0)=(A_{\Psi}f,0), \quad f\in \mathcal{D}(A_{\Psi}).
\]
From Theorem~\ref{th:main} it follows that there exists an extension $\G$ of $(\mathcal{A}+\mathcal{B})_{|}$ generating a substochastic semigroup on $\BS_0$ and with resolvent operator of $\G$ at $\lambda>0$ given by \eqref{d:resgG}. We have $\G(f,0)=(G_{\Psi}f,0)$ for $(f,0)\in \mathcal{D}(\G)=\mathcal{D}(G_{\Psi})\times \{0\}$. Hence, $G_{\Psi}f=A_{\Psi}f$ for $f\in \mathcal{D}(A_{\Psi})$ and $G_{\Psi}$ is the generator of a substochastic semigroup on $L^1$. Since $R(\lambda,\A)\B (f,0)=(\Psi(\lambda)\Psi f,0)$ for $f\in \mathcal{D}$, by \eqref{eq:RA} and \eqref{e:opB1},  we obtain $(R(\lambda,\A)\B)^n (f,0)=((\Psi(\lambda)\Psi)^n f,0)$ for all $n\ge 1$ and $f\in \mathcal{D}$ implying that
\[
\begin{split}
R(\lambda,\A)(\B R(\lambda,\A))^n(f,0)&=((\Psi(\lambda)\Psi)^n R(\lambda,A_0)f,0)\\
&=(\Psi(\lambda) (\Psi\Psi(\lambda))^{n-1}\Psi R(\lambda,A_0)f,0).
\end{split}
\]
This together with \eqref{d:resgG} shows that $R(\lambda,G_{\Psi})$ is given by \eqref{e:GPsi}. It remains to show that the operator $(A,\mathcal{D})$ is an extension of $G_{\Psi}$. To this end we take $f\in L^1$, $g=R(\lambda,G_\Psi)f$ and
\[
g_N=R(\lambda,A_0)f+\sum_{n=0}^N \Psi(\lambda)(\Psi \Psi(\lambda))^n \Psi R(\lambda,A_0)f, \quad N\ge 0.
\]
We have $g_N\to g$ in $L^1$ as $N\to \infty$, by \eqref{e:GPsi}.
Since $R(\lambda,A_0)f\in \mathcal{D}$ and $\Psi(\lambda)$ has values in $\mathcal{D}$, we get
$g_N\in \mathcal{D}$ and $Ag_N=\lambda g_N-f$. Thus $Ag_N\to \lambda g-f$ in $L^1$ as $N\to \infty$. Since the operator $\A$ is closed, we see that the operator $A$ is closed. Thus $g\in \mathcal{D}$ and $\lambda g-f= Ag$. Consequently, $Ag=\lambda g -(\lambda g-G_\Psi g)=G_\Psi g$, completing the proof of \eqref{d:GpsiA}.

To prove the equivalence in \eqref{i:ec1} we make use of Theorem~\ref{th:spp} \eqref{i:1}. We have
\[
(\mathcal{I}-\B R(\lambda,\A))(f,f_\partial)=(f, -\Psi R(\lambda,A_0)f+(I_{\partial}-\Psi \Psi(\lambda))f_{\partial})
\]
for all $(f,f_{\partial})\in L^1\times L^1_{\partial}$.
Hence,
$L^1\times \{0\}\subseteq (\mathcal{I}-\B R(\lambda,\A))(L^1\times L^1_{\partial})$ if and only if for each $f\in L^1$ there exists $f_{\partial}\in L^1_{\partial}$ such that
\[
\Psi R(\lambda,A_0)f=(I_{\partial}-\Psi \Psi(\lambda))f_{\partial},
\]
completing the proof of \eqref{i:ec1}, since $\Ker(\Psi_0)=\mathcal{D}(A_0)=R(\lambda,A_0)(L^1)$.

Finally, observe that for each $f\in L^1$ we have
\[
\lim_{n\to \infty}(\B R(\lambda,\A))^n(f,0)=(0,0) \Leftrightarrow \lim_{n\to \infty}(\Psi\Psi(\lambda))^n \Psi R(\lambda,A_0)f=0.
\]
Thus, Theorem~\ref{th:equiv} completes the proof.
\end{proof}

\begin{remar}
Note that it follows from the proof of Theorem~\ref{th:main} that for each $r\in(0,1)$ the operator $(A_{r\Psi},\mathcal{D}(A_{r\Psi}))$ generates a substochastic semigroup $\{S_r(t)\}_{t\ge 0}$ on $L^1$, for any $t\ge 0$ and $f\in L^1$ the limit
$S_{\Psi}(t)f=\lim_{r\to 1}S_r(t)f$ exists in $L^1$ and defines a substochastic semigroup $\{S_{\Psi}(t)\}_{t\ge 0}$ and that $(G_{\Psi},\mathcal{D}(G_{\Psi}))$   is the generator of $\{S_{\Psi}(t)\}_{t\ge 0}$.
\end{remar}

\begin{corollar}\label{c:coro1}
Under the assumptions of Theorem~\ref{th:pertboun} the following holds:
\begin{enumerate}[\upshape(i)]
\item\label{it:inv} If there is $\lambda>0$ such that $(I_{\partial}-\Psi \Psi(\lambda))(L^1)=L^1_{\partial}$  then the operator $(A_{\Psi},\mathcal{D}(A_{\Psi}))$ as in \eqref{d:dombou} is the generator of a substochastic semigroup on $L^1$.

\item If there is $\lambda>0$ such that $(\Psi\Psi(\lambda))^nf_{\partial}\to 0$ as $n\to \infty$ for all $f_{\partial}\in L^1_{\partial}$ then the closure of the operator $(A_{\Psi},\mathcal{D}(A_{\Psi}))$ is the generator of a substochastic semigroup.
\end{enumerate}
\end{corollar}

\begin{remar}
In  \cite{GMTKpositive}  we assumed   that there exists $\omega\in \mathbb{R}$ such that the operator $I_{\partial}-\Psi \Psi(\lambda)\colon L^1_{\partial}\to L^1_{\partial}$ is invertible with positive inverse for all $\lambda>\omega$.
Since $\Psi \Psi(\lambda)$ is a positive operator with $\ker(I_{\partial}-\Psi \Psi(\lambda))=\{0\}$,  the operator $I_{\partial}-\Psi \Psi(\lambda)$ is invertible with positive inverse if and only if
the spectral radius of the operator $\Psi \Psi(\lambda)$ is strictly smaller than 1,
or equivalently, see Remark~\ref{r:spectral},
\begin{equation}\label{eq:ssper}
\lim_{n\to \infty}\|(\Psi \Psi(\lambda))^n\|=0.
\end{equation}
\end{remar}

We conclude this section with the perturbation result from \cite{GMTK} being a consequence of Theorem~\ref{th:main}. Making use of Theorem~\ref{th:equiv}  we also obtain a condition for the closure property.

 \begin{theore}\label{th:pertbounB}
Let the operators $A$ and $\Psi_0$ satisfy Assumption \ref{a:S2} and let $(A,\Ker(\Psi_0))$ be the generator of a substochastic semigroup on $L^1$.
Suppose that $B\colon\mathcal{D}\to L^1$ and $\Psi\colon \mathcal{D}\to L^1_\partial$ are positive operators such that
\begin{equation}\label{eq:zero}
\int_{E} (Af +Bf) \, dm+\int_{E_{\partial}}(\Psi f-\Psi_0f)\, dm_{\partial}\le 0\quad \text{for all nonnegative } f\in \mathcal{D}.
\end{equation}
Then there exists an extension $G$ of $A_{\Psi}+B$ with $A_{\Psi}$ as in \eqref{d:dombou} generating a substochastic semigroup on $L^1$. We have $\mathcal{D}(A_\Psi)\subseteq \mathcal{D}(G)$ and
the resolvent operator of $G$ at $\lambda>0$ is given by
\begin{equation}\label{d:resg}
R(\lambda,G)f=\sum_{n=0}^{\infty}(R(\lambda,A_0)B+\Psi(\lambda)\Psi)^n R(\lambda,A_0)f,\quad f\in L^1,
\end{equation}
Moreover,
$G=\overline{A_{\Psi}+B}$ if and only if
\begin{equation}
\lim_{n\to \infty} B (R(\lambda,A_0)B+\Psi(\lambda)\Psi)^nf= 0 \quad \text{and} \quad \lim_{n\to \infty} \Psi (R(\lambda,A_0)B+\Psi(\lambda)\Psi)^nf= 0
\end{equation}
for all $f\in \Ker(\Psi_0)$ and for some/all $\lambda>0$.
\end{theore}
\begin{proof}
Let  the operators $\A$ and $\B$ be as in \eqref{d:AB0}. Then all assumptions of Theorem~\ref{th:main} hold.
We have $R(\lambda,\A)(L^1\times L^1_{\partial})=\mathcal{D}\times \{0\}$ and
\[
R(\lambda,\A)\B (f,0)=((R(\lambda,A_0)B+\Psi(\lambda)\Psi)f,0),\quad f\in \mathcal{D}.
\]
Thus the formula for $R(\lambda,G)$ follows from \eqref{d:resg} and the characterization of the closure property follows from condition \eqref{e:merg}.
\end{proof}

Theorem~\ref{th:pertbounB} together with Theorem \ref{th:equiv} implies the following result.
\begin{corollar}\label{c:petbounB}
Let the operators $A$ and $\Psi_0$ satisfy Assumption \ref{a:S2} and let $(A,\Ker(\Psi_0))$ be the generator of a substochastic semigroup.
Suppose that $B\colon\mathcal{D}\to L^1$ and $\Psi\colon \mathcal{D}\to L^1_\partial$ are positive operators such that
\begin{equation}\label{eq:zeroo}
\int_{E} (Af +Bf) \, dm+\int_{E_{\partial}}(\Psi f-\Psi_0f)\, dm_{\partial}= 0\quad \text{for all nonnegative } f\in \mathcal{D}.
\end{equation}
Then the substochastic semigroup $\{P(t)\}_{t\ge 0}$ on $L^1$ generated by an extension $G$ of $A_{\Psi}+B$ with $A_{\Psi}$ as in \eqref{d:dombou} is stochastic if and only if $G=\overline{A_{\Psi}+B}$.
\end{corollar}

\section{Transport equations with conservative boundary conditions}\label{s:4}

\subsection{General assumptions}

Here we consider the general setting for PDMPs as introduced in \cite{GMTK}. Let $\widetilde{E}$ be a separable metric space and let $\phi=\{\phi_t\}_{t\ge 0}$  be a flow  on $\widetilde{E}$.
Our basic assumption is the following.
 \begin{assumptio} \label{a:nons}
There exists a \emph{measurable cocycle} $\{J_t\}_{t\in \mathbb{R}}$ of $\phi$ on $\widetilde{E}$, i.e. a family of Borel measurable nonnegative functions  satisfying the following conditions
\begin{equation*}\label{d:jacob}
J_0(x)=1,\quad J_{t+s}(x)=J_{t}(\phi_{s}(x))J_{s}(x),\quad s,t\in \mathbb{R}, x\in \widetilde{E},
\end{equation*}
and there exists a $\sigma$-finite  Radon measure $m$ on the Borel $\sigma$-algebra $\mathcal{B}(\widetilde{E})$ with $m(\partial E)=0$ such that
\begin{equation*}\label{d:RND}
(m\circ \phi_t^{-1})(B )=m(\phi_t^{-1}(B ))=\int_{B}  J_{-t}(x)m(dx),\quad t\in \mathbb{R}, B \in \mathcal{B}(\widetilde{E}).
\end{equation*}
\end{assumptio}

\begin{remar}
Assumption \ref{a:nons} implies that for each $t$ the transformation $\phi_t\colon \widetilde{E}\to \widetilde{E}$ is \emph{non-singular with respect to the measure $m$} (\cite{almcmbk94}), i.e. $m\circ \phi_t^{-1}$ is absolutely continuous with respect to $m$. Then $J_{-t}$ is the Radon-Nikodym derivative $\frac{d m\circ \phi_t^{-1}}{dm}$.
\end{remar}

\begin{remar} Consider $\widetilde{E}=\mathbb{R}^N$ as in the Introduction and a flow $\phi$ solving \eqref{e:vjp}.
If we take as $m$ the Lebesgue measure on $\mathbb{R}^N$ then  $J_t(x)$  is the absolute value of the determinant of the derivative of the mapping $x\mapsto \phi_t(x)$, by the change of variables formula.
By Liouville's theorem, it is also given by
\begin{equation}\label{d:Liou}
J_{t}(x)
=\exp\left\{\int_0^{t} a(\phi_{r}(x))dr\right\},\quad  t\in \mathbb{R},
\end{equation}
where $a$ is the divergence of  $b$. In particular, if $b$ is globally Lipschitz 
then
the function $a$ is the divergence of the vector field $b$.
Note that in \cite{arlottibanasiaklods09} it is assumed that there exists a Radon measure $m$ on $\mathbb{R}^N$ that is invariant for the flow $\{\phi_t\}_{t\in \mathbb{R}}$, i.e. $m(\phi_t^{-1}(B))=m(B)$ for all Borel subsets of $\mathbb{R}^N$ and all $t\in \mathbb{R}$. This corresponds to $J_t(x)\equiv 1$ in Assumption \ref{a:nons}, so that $a\equiv 0$ and we have the divergence free case.
\end{remar}

We define the \emph{hitting times} of the boundaries $\Gamma^{\pm}$ by
\begin{equation*}\label{d:exit}
t_{\pm}(x)=\inf\{t>0:\phi_{\pm t}(x)\in \Gamma^{\pm} \} 
\end{equation*}
with the convention that $\inf\emptyset=\infty$. We set $t_{\pm}(x)=0$ for $x\in \Gamma^{\pm} $ and we extend the above formula to  points from the boundaries $\Gamma^{\mp}$.
It is shown in \cite{arlottibanasiaklods07} that if $\widetilde{E}=\mathbb{R}^N$ and  $a$ in \eqref{d:Liou} is bounded then there exist unique Borel measures satisfying the following.
\begin{assumptio} \label{a:dive}
There exist  Borel measures $m^{\pm}$ on $\Gamma^{\pm} $ such that
 for any nonnegative and Borel measurable $f$, we have
\begin{equation*}
 \int_{E_+}f(x)\,m(dx)=\int_{\Gamma^{+} }\int_{0}^{t_{-}(z)} f(\phi_{-s}(z))J_{-s}(z)\,ds\, m^{+}(dz)
\end{equation*}
and
\begin{equation*}
 \int_{E_{-}}f(x)\,m(dx)=\int_{\Gamma^{-} }\int_{0}^{t_{+}(z)} f(\phi_s(z))J_s(z)\,ds\, m^{-}(dz),
\end{equation*}
where
\begin{equation*}
E_{\pm}=\{x\in E: 0<t_{\pm}(x)<\infty\}\quad \text{and}\quad E=E^0\cup \Gamma^{-}\setminus \Gamma^{-}\cap\Gamma^{+}.
\end{equation*}
\end{assumptio}

As concern the Borel measurable function $q\colon E\to [0,\infty)$ we additionally impose the following.
\begin{assumptio} \label{a:varp}
For each $x$  the function $\mathbb{R}\ni t\mapsto \int_0^tq(\phi_{r}(x))dr$ is absolutely continuous, where we
extend $q$ beyond $E$  by setting $q(x)=0$ for $x\not\in E$.
\end{assumptio}

\subsection{Existence of solutions}

We first consider well-posedness of \eqref{e:eq} with $\Psi\equiv 0$ and $B\equiv 0$. 
To describe the transport operator we use the approach of~\cite{arlottibanasiaklods07,arlottibanasiaklods09} as extended in \cite{GMTK}.
Let $\mathfrak{N}$ be the set of all measurable and bounded functions $\psi  \colon E\to \mathbb{R}$ with compact support in $E^0$ and such that for any $x\in E$ the function
\[
(-t_{-}(x),t_{+}(x))\ni t\mapsto \psi  (\phi_t(x))
\]
is continuously differentiable with bounded and measurable derivative at $t=0$, i.e. the mapping
\[
 x\mapsto \frac{d(\psi  \circ \phi_t)}{dt}\Big|_{t=0}(x)
\]
is bounded and measurable.
We define  the \emph{maximal transport operator} $T_{\max}$ on a set $\mathcal{D}_{\max}\subseteq L^1(E,m)$ as follows.  We say that
$ f\in \mathcal{D}_{\max}$ if there exists $g\in L^1(E,m)$ such that
\begin{equation}\label{d:Tm}
\int_{E} g(x)\psi  (x)m(dx)=\int_{E} f(x)\frac{d(\psi  \circ \phi_t)}{dt}\Big|_{t=0}(x)m(dx)
\end{equation}
for all $\psi  \in \mathfrak{N}$ and we set $T_{\max} f:=g$.
If Assumptions~\ref{a:nons} hold  and if $f\in \mathcal{D}_{\max}$ then there exists  a representative $f^\sharp$ of $f$ such that for $m$-a.e. $x\in E$ and any $-t_{-}(x)<t_1\le t_2<t_{+}(x)$ we have
\begin{equation*}
f^{\sharp}(\phi_{t_1}(x))J_{t_1}(x)-f^{\sharp}(\phi_{t_2}(x))J_{t_2}(x)=\int_{t_1}^{t_2}T_{\max} f(\phi_s(x))J_s(x)\,ds.
\end{equation*}

Given  $f\in L^1(E,m)$ we define its traces $\Tr^{\pm}f$ on the boundaries $\Gamma^{\pm} $ by the
the pointwise limits
\begin{equation}\label{d:traces}
\Tr^{\pm}f(z)=\lim_{s\to 0^{+}}f(\phi_{\mp s}(z))J_{\mp s}(z)
\end{equation}
provided that the limits exist for $m^{\pm}$-a.e. $z\in \Gamma^{\pm} $. It can be shown~\cite{arlottibanasiaklods09,GMTK} that $\Tr^{\pm}f$ exist for $f\in \mathcal{D}_{\max}$. If $\Gamma^{\pm} =\emptyset$ then we set $\Tr^{\pm}= 0$.   We write
\[
\mathcal{D}(\Tr^{\pm})=\{f\in L^1(E,m):\Tr^{\pm}f\in L^1(\Gamma^{\pm} ,m^{\pm})\}.
\]
Note that the traces $\Tr^{\pm}\colon \mathcal{D}(\Tr^{\pm})\to L^1(\Gamma^{\pm} ,m^{\pm})$ are linear positive operators.

The following result corresponds to   Green's identity as in \cite[Proposition 4.6]{arlottibanasiaklods07} and its proof is given in \cite{GMTK}. Formula \eqref{e:Afin0} explains the interplay between the transport operator, the boundary measures and the traces, giving conservation of mass.

\begin{theore}\label{l:Green0} Suppose that Assumptions~\ref{a:nons}--\ref{a:dive} hold. Let $(T_{\max},\mathcal{D}_{\max})$ be the maximal transport operator  as in \eqref{d:Tm}.
If $f\in \mathcal{D}_{\max}$ is such that  $\Tr^{-}f\in L^1(\Gamma^{-} ,m^-)$ then  $\Tr^+f\in L^1(\Gamma^{+} ,m^+)$
 and
\begin{equation}\label{e:Afin0}
\int_{E} T_{\max} f(x)m(dx)=\int_{\Gamma^{-} }\Tr^{-}f(x)m^{-}(dx)-\int_{\Gamma^{+} }\Tr^{+}f(x)m^{+}(dx).
\end{equation}
\end{theore}

We now define the operator $A\colon \mathcal{D}\to L^1(E,m)$ by
\begin{equation}\label{d:operatorA}
Af =T_{\max} f -q f,\quad f\in \mathcal{D},
\end{equation}
where the transport operator $T_{\max}$ is as in \eqref{d:Tm}, $q\colon E\to [0,\infty)$ is a Borel measurable function and
\begin{equation}\label{d:domainopA}
\mathcal{D}=\{f\in \mathcal{D}_{\max}: \Tr^{-}f \in L^1(\Gamma^{-} ,m^{-}), q f\in L^1(E,m)\}.
\end{equation}
The next result shows that a restriction of the operator $A$ is the generator of a substochastic semigroup and that Assumption \ref{a:S2} holds. Its proof is given in \cite{GMTK}.

\begin{theore}\label{t:gst} Suppose that Assumptions~\ref{a:nons} and \ref{a:varp} hold. Let $(A,\mathcal{D})$ be as in \eqref{d:operatorA}--\eqref{d:domainopA} and let $\Psi_0(f)=\Tr^{-}f$ for $f\in \mathcal{D}$.
Then the operator $(A_0,\mathcal{D}(A_0))=(A,\Ker(\Psi_0))$
is the generator of the substochastic semigroup $\{S(t)\}_{t\ge 0}$ on $L^1(E,m)$ given by
\begin{equation*}
S(t)f(x)=\mathbf{1}_{E}(\phi_{-t}(x))f(\phi_{-t}(x))J_{-t}(x)e^{-\int_0^t q(\phi_{-r}(x))dr}
\end{equation*}
for $t>0$, $x\in E$, $f\in L^1(E,m)$.
For each $\lambda>0$  the right-inverse $\Psi(\lambda)$ of the operator $\Psi_0$  restricted to $\Ker(\lambda -A)$ is a positive operator on $L^{1}(\Gamma^{-} ,m^{-})$ of the form
\begin{equation}\label{e:psilam}
\Psi(\lambda)f_{\partial}(x)=e^{-\lambda t_{-}(x)-\int_0^{t_{-}(x)}q(\phi_{-r}(x))dr}f_{\partial}(\phi_{-t_{-}(x)}(x))J_{-t_{-}(x)}(x),\quad x\in E,
\end{equation}
$f_{\partial}\in L^{1}(\Gamma^{-} ,m^{-})$
and
for any  $\lambda>0$  we have $\Tr^{+}\Psi(\lambda)f_{\partial}\in  L^1(\Gamma^{+},m^+)$ for $f_\partial \in L^1(\Gamma^{-} ,m^{-})$ with
\begin{equation}\label{e:gppsilam}
\Tr^{+}\Psi(\lambda)f_{\partial}(z)=e^{-\lambda t_{-}(z)-\int_0^{t_{-}(z)} q(\phi_{-r}(z))dr}f_{\partial}(\phi_{-t_{-}(z)}(z))J_{-t_{-}(z)}(z),\quad z\in \Gamma^{+}.
\end{equation}
Moreover, if $f\in L^1(E,m)$ then
\[
R(\lambda,A_0)f(x)=\int_0^{t_{-}(x)}e^{-\lambda t-\int_0^{t}q(\phi_{-r}(x))dr}f(\phi_{-t}(x))J_{-t}(x)dt, \quad x\in E,
\]
 and $\Tr^{+}R(\lambda,A_0)f\in  L^1(\Gamma^{+},m^+)$ with
\[
\Tr^{+}R(\lambda,A_0)f(z)=\int_{0}^{t_{-}(z)}e^{-\lambda t-\int_0^{t}q(\phi_{-r}(z))dr}f(\phi_{-t}(z))J_{-t}(z)dt,\quad z\in \Gamma^{+} , f\in L^1(E,m).
\]
\end{theore}

\begin{remar}
In  the notation used in \cite{arlottibanasiaklods09,arlottibanasiaklods11,arlottilods14} we have  $J_t(x)\equiv 1$ and $q\equiv 0$. The operators appearing in Theorem~\ref{t:gst} are denoted there, respectively, by
$\Xi_\lambda=\Psi(\lambda)$, $M_\lambda=\Tr^+\Psi(\lambda)$, $C_\lambda=R(\lambda,A_0)$, and $G_\lambda=\Tr^{+}R(\lambda,A_0)$.
\end{remar}

We conclude this section with the following.

\begin{theore}\label{th:exsto}
Suppose that Assumptions~\ref{a:nons}--\ref{a:varp} hold.  Let $(A,\mathcal{D})$ be defined by~\eqref{d:operatorA}--\eqref{d:domainopA},  $\Psi_0(f)=\Tr^{-}f$ for $f\in \mathcal{D}$ and let
$B\colon \mathcal{D}\to L^1(E,m)$ and $\Psi\colon \mathcal{D}\to L^1(\Gamma^{-} ,m^{-})$ be positive operators satisfying
\begin{equation}\label{e:jump2}
\int_{E} B f dm+ \int_{\Gamma^{-} } \Psi f dm^{-} \le \int_{E} qf dm +\int_{\Gamma^{+} } \Tr^{+} fdm^{+}
\end{equation}
for all nonnegative $ f\in \mathcal{D}.$ Then there exists a substochastic semigroup $\{P(t)\}_{t\ge 0}$ with
generator $(G,\mathcal{D}(G))$ satisfying \eqref{d:resg} and
\[
Gf=A_{\Psi}f+Bf,\quad A_{\Psi}f=Af,\quad  f\in \mathcal{D}(A_{\Psi})=\{f\in \mathcal{D}:\Tr^{-}f=\Psi f\}.
\]
\end{theore}
\begin{proof}
Condition \eqref{e:jump2} combined with \eqref{e:Afin0}  implies that \eqref{eq:zero} holds with $\Psi_0=\Tr^{-}$. It follows from
Theorem~\ref{t:gst} that Theorem~\ref{th:pertbounB} applies giving the existence of a substochastic semigroup with generator extending the operator $A_{\Psi}+B$.
\end{proof}

\subsection{Applications to PDMPs}

Let $(\phi,q,\mathcal{P})$ be the characteristics of the minimal Markov process $\{X(t)\}_{t\ge 0}$.
We say that the minimal process $X=\{X(t)\}_{t\ge 0}$ \emph{induces a substochastic semigroup} $\{P(t)\}_{t\ge 0}$ on $L^1(E,m)$ if
\begin{equation}\label{d:indss}
\int_{B }P(t)f(x)m(dx)=\int_{E}\Pr(X(t)\in B ,t<\tau_{\infty}|X_0=x)f(x)m(dx)
\end{equation}
for all $f\in L^1(E,m)$, $B \in \mathcal{B}(E)$, $t>0$.
Suppose that Assumptions~\ref{a:nons}--\ref{a:dive} hold.
Recall that $L^1(\Gamma^{-},m^{-})$ and $L^1(\Gamma^{+},m^{+})$ are the trace spaces corresponding to the boundaries $\Gamma^{-}$ and $\Gamma^{+}$.
The jump distribution $\mathcal{P}$  is assumed to be non-singular in the following sense:
there exists a positive operator
$(P_0,P_{\partial})\colon L^1(E,m)\times L^1(\Gamma^{+},m^{+})\to L^1(E,m)\times L^1(\Gamma^{-}, m^{-})$
such that
\begin{multline}\label{e:jump}
\int_{E} \jump (x,B )f(x)m(dx)+\int_{\Gamma^{+} } \jump (x,B )f_{\partial^+}(x)m^{+}(dx) \\=
\int_{B \cap E^0} P_{0}(f,f_{\partial^+})(x)m(dx)+ \int_{B \cap \Gamma^{-} } P_{\partial}(f,f_{\partial^+})(x)m^{-}(dx),
\end{multline}
for all Borel subsets $B$ of $E$ and all nonnegative $f\in L^1(E,m)$,  $f_{\partial^+}\in L^1(\Gamma^{+} ,m^{+})$.  
Our next result is  an extension of Theorem 2.8 in \cite{GMTK}. It follows from Theorem \ref{th:exsto},  Corollary~\ref{c:petbounB} and  the first part of Theorem 2.8 in~\cite{GMTK}.
\begin{theore}\label{th:ptsmal} Suppose that Assumptions~\ref{a:nons} to \ref{a:varp} hold.  Let $(A,\mathcal{D})$ be defined by~\eqref{d:operatorA}--\eqref{d:domainopA},  $\Psi_0(f)=\Tr^{-}f$ for $f\in \mathcal{D}$  and let
$B\colon \mathcal{D}\to L^1(E,m)$, $\Psi\colon \mathcal{D}\to L^1(\Gamma^{-} ,m^{-})$ be given  by
\begin{gather}\label{d:part}
Bf=P_{0}(q f, \Tr^+  f),\quad \Psi f=P_{\partial} (q f, \Tr^{+}f), \quad f\in \mathcal{D},
\end{gather}
where
$P_0$, $P_{\partial}$ satisfy \eqref{e:jump}.
Then the minimal process $X$ with characteristics $(\phi,q,\jump)$ induces a substochastic semigroup $\{P(t)\}_{t\ge 0}$ with
generator $(G,\mathcal{D}(G))$ satisfying \eqref{d:resg} and being an extension of the operator $A_{\Psi}+B$.

Moreover,
$\{P(t)\}_{t\ge 0}$ is stochastic if and only if $G=\overline{A_{\Psi}+B}$.
\end{theore}

\begin{remar}  We do not have to define $\jump(x,\cdot)$ at every point $x\in E\cup \Gamma^{+}$ as a probability measure on $E$, i.e. $\jump(x,E)=1$.
Instead, we only need that $q$ and $\jump$ satisfy the following 
 \begin{equation*}
 \jump(x,E)q(x)=q(x)  \quad \text{for }x\in E\quad \text{and}\quad \jump(x,E)=1 \quad \text{for }x\in \Gamma^{+}.
 \end{equation*}
\end{remar}

We now provide a useful criterion for the generator to be the closure by using Theorem~\ref{t:opK}.
The substochastic operator $K\colon L^1(E,m)\times L^1(\Gamma^{-},m^{-})\to   L^1(E,m)\times L^1(\Gamma^{-},m^{-})$ as defined in \eqref{eq:K} is of the form~\cite{GMTK}
\begin{equation}\label{d:operatorK}
K (f,f_{\partial})=\big(P_0(q R_0(f,f_\partial),R_0(f,f_\partial)),P_{\partial}(q R_0(f,f_\partial),R_0(f,f_\partial))\big),
\end{equation}
where
$P_0,P_\partial$ satisfy \eqref{e:jump} and
\begin{multline}\label{d:R0p}
R_0(f,f_{\partial})(x)=\int_{0}^{t_{-}(x)}e^{-\int_0^t q(\phi_{-r }(x))dr} f(\phi_{-t}(x))J_{-t}(x)dt \\ +\mathbf{1}_{\{t_{-}(x)<\infty\}}e^{-\int_0^{t_{-}(x)} q(\phi_{-r }(x))dr} f_\partial(\phi_{-t_{-}(x)}(x))J_{-t_{-}(x)}(x),\quad x\in E\cup \Gamma^{+},
\end{multline}
for nonnegative $(f,f_{\partial})\in L^1(E,m)\times L^1(\Gamma^{-} ,m^{-})$.

\begin{corollar}\label{c:closure} Let $R_0$ be defined in \eqref{d:R0p}. In the setting of Theorem~\ref{th:ptsmal}  suppose that one of the following holds:
\begin{enumerate}[\upshape(i)]
\item There exist quasi-interior elements $f\in L^1(E,m)$ and $f_{\partial}\in L^1(\Gamma^{-},m^{-})$ such that
\begin{equation}\label{e:qiP01}
B(R_0(f,f_\partial))\le f \quad \text{and}\quad  \Psi(R_0(f,f_\partial))\le f_{\partial}.
\end{equation}
\item $B\equiv 0$ and there exists a quasi-interior element $f_{\partial}\in L^1(\Gamma^{-},m^{-})$ such that \linebreak[4]$\Psi(R_0(0,f_\partial))\le f_{\partial}.$
\end{enumerate}
Then  $G=\overline{A_{\Psi}+B}$.
\end{corollar}
\begin{proof}
If the first condition holds then the operator $K$ defined in \eqref{d:operatorK} is mean ergodic on $L^1(E,m)\times L^1(\Gamma^{-},m^{-})$. Therefore, Theorems~\ref{t:opK} and \ref{th:equiv} imply the closure property.
Now if the second condition holds, then $K(L^1(E,m)\times \{0\})\subseteq \{0\}\times L^1(\Gamma^{-},m^{-})$ and the substochastic operator $K_{\partial}\colon L^1(\Gamma^{-},m^{-})\to L^1(\Gamma^{-},m^{-})$ defined by
\[
K_{\partial}f_{\partial}=\Psi\big( R_0(0,f_\partial)\big),  \quad f_{\partial}\in L^1(\Gamma^{-},m^{-}),
\] is mean ergodic on $L^1(\Gamma^{-},m^{-})$ implying that $K$ is mean ergodic on $L^1(E,m)\times \{0\}$.
\end{proof}

\subsection{Collisionless kinetic equations}

In this section we suppose that Assumptions~\ref{a:nons}--\ref{a:varp} hold and $B\equiv 0$.
Our Theorem~\ref{th:pertboun} extends the generation results for streaming operators with abstract boundary conditions obtained in \cite{arlottilods05,mustapha08} for the free transport equation and in \cite{arlottibanasiaklods11,arlottilods14} for the divergence free vector fields.
Assume now that the boundary operator $\Psi$ is of the form $\Psi(f)=H( \Tr^{+} f)$
where $H\colon L^1(\Gamma^{+},m^{+})\to L^1(\Gamma^{-}, m^{-})$ is a stochastic operator then
our Theorem \ref{th:pertboun} implies \cite[Theorem 2.5]{arlottilods14} and \cite[Theorem 6.2]{arlottibanasiaklods11}.
Moreover, Theorem~\ref{th:pertboun} \eqref{i:ec1} gives the following.
\begin{corollar}\label{c:pert1}
If $\Psi=H\Tr^+$ and one of the following holds
\begin{enumerate}[\upshape(i)]
\item $\mathrm{ess}\inf\{t_{+}(z):z\in \Gamma^{-}\}>0$,
\item $(I-\Tr^{+}\Psi(\lambda) H)(L^1(\Gamma^{+},m^{+}))=L^1(\Gamma^{+},m^{+})$ for some $\lambda>0$,
\end{enumerate}
then $A_{\Psi}$ is the generator.
\end{corollar}
\begin{proof}
For any nonnegative $f_\partial \in L^1(\Gamma^{-} ,m^{-})$, it holds
\begin{equation*}
\int_{\Gamma^{+} }\Tr^{+}\Psi(\lambda)f_\partial(z)m^{+}(dz)=\int_{\Gamma^{-} }e^{-\lambda t_{+}(z)-\int_0^{t_{+}(z)}q(\phi_{r}(z))dr}f_{\partial}(z) m^{-}(dz).
\end{equation*}
Thus the operator $\Tr^{+}\Psi(\lambda)$ has norm less than $e^{-\lambda c}$ for some positive constant $c$, implying that
$\|H\Tr^{+}\Psi(\lambda)\|<1$. The second condition implies that $\Psi (\Ker(\Tr^{-}))\subseteq (I_{\partial}-\Psi\Psi(\lambda))(L^1(\Gamma^{-},m^{-})$.
\end{proof}

It follows from Theorem~\ref{th:pertboun} part~\eqref{i:ec2} that the operator
$\overline{A_{\Psi}}$  with $\Psi=H\Tr^+$ is the generator if and only if
\[
\lim_{n\to \infty}\|(H\Tr^{+}\Psi(\lambda))^n H\Tr^{+}R(\lambda,A_0)f\|=0
\]
for all $f\in L^1(E,m)$ and some $\lambda>0$. Since the operators $H$ and $\Tr^{+}\Psi(\lambda)$ are bounded, this is equivalent to
\[
\lim_{n\to \infty}\|(\Tr^{+}\Psi(\lambda) H)^n\Tr^{+}R(\lambda,A_0)f\|=0
\]
for all $f\in L^1(E,m)$ and some $\lambda>0$, recovering the corresponding results \cite[Lemma 6]{mustapha08} and \cite[Proposition 6.2]{arlottibanasiaklods11}. Note that if $q\equiv 0$ and $J_t\equiv 1$ then $\Tr^{+}$  is surjective by  \cite[Proposition 2.3]{arlottibanasiaklods11}.

Finally, our next result extends \cite[Proposition 3.5]{arlottilods05},  \cite[Theorem 21]{mustapha08} and \cite[Proposition 8]{arlottilods14}.

\begin{corollar}\label{c:perturb2}
If $\Psi=H\Tr^+$ and there exists a quasi-interior element $f_{\partial+}\in L^1(\Gamma^{+},m^{+})$ such that
\begin{equation}\label{e:exqi}
H(f_{\partial+})(\phi_{-t_{-}(x)}(x))J_{-t_{-}(x)}(x)e^{-\int_0^{t_{-}(x)} q(\phi_{-r }(x))dr} \mathbf{1}_{\{t_{-}(x)<\infty\}}\le f_{\partial+}(x) ,\quad x\in \Gamma^{+},
\end{equation}
then $\overline{A_{\Psi}}$ is the generator.
\end{corollar}
\begin{proof}
The assumption \eqref{e:exqi}  implies that for any $\lambda>0$ the operator
$
\Tr^+\Psi(\lambda)H
$ is mean ergodic on $L^1(\Gamma^{+},m^{+})$. Thus for any $g\in L^1(\Gamma^{+},m^{+})$ the sequence
\[
\frac{1}{N}\sum_{n=0}^{N-1} (\Tr^+\Psi(\lambda)H )^n g
\]
converges in $L^1(\Gamma^{+},m^{+})$. Since the operator $H$ is continuous, we conclude that the operator $\Psi\Psi(\lambda)=H\Tr^+\Psi(\lambda)$ is mean ergodic on $H(L^1(\Gamma^{+},m^{+}))$. Consequently, $(\Psi\Psi(\lambda))^n f_{\partial}\to 0$ for all $f_{\partial}\in H(L^1(\Gamma^{+},m^{+}))$. Since $\Psi(\Ker(\Tr^{-})\subseteq H(L^1(\Gamma^{+},m^{+}))$, the result follows from Theorem~\ref{th:pertboun}~\eqref{i:ec2}.
\end{proof}

\section{Examples}

In this section we illustrate our abstract results with particular examples of processes for which one can check that the induced semigroup is stochastic.

\begin{exampl}[Gene expression with bursting and memory]
Gene expression is a process by which the information from a gene is used to synthesize
proteins.
Proteins are  basic components of
living organisms. They are polymers made of amino acids. In nature
there are 20 different amino acids.
The amino acid sequence in proteins is constant in a given species and genetically encoded.
Patterns are stored in DNA.
The gene is expressed if it is prescribed by RNA polymerase (transcribed) from DNA to messenger RNA (mRNA), ribosomes bind to the transcribed mRNA and synthesize the protein in the translation process.
Only part of the genes in the cell is expressed at any given time.

Gene expression is inherently stochastic which is the effect of the low copy numbers of DNA
and can lead to large variability in molecule levels for genetically identical cells.
In experimental studies \cite{cai}, it was observed that the synthesis of proteins is at random time intervals and in random amounts, characterized by the occurrence of \emph{bursts}  (\emph{translational bursts}).  Similarly, it has been observed in
\cite{golding}, that mRNA can also be produced in the form of bursts (\emph{transcriptional bursts}).  We allow translations/transcripts
to effectively be made both in arbitrary independent
bursts  and at arbitrary independent time intervals as observed in \cite{pedraza08,kumar2015transcriptional}.

We model the amount $\x(t)$ of molecules (mRNA or protein)  in a cell at time $t$ as a  continuous variable. We assume that molecules undergo degradation with rate $\gamma>0$,  that a random amount $\eta_n$ of molecules is produced  through bursting at random time $\tau_n$, $n\ge 1$, and  that $\eta_n$ and $T_n=\tau_n-\tau_{n-1}$, where $\tau_0=0$, are independent random variables with densities $h$ and  $h_T$, respectively. If $h_T$  is exponential then we recover the models from \cite{friedman06,mcmmtkry13} with constant intensity. To study our model as a PDMP we introduce the variable $x=(\x,\mathrm{s})$, where $\mathrm{s}$ denotes the time that has elapsed since the last occurrence of bursts. We have $b(\x,\mathrm{s})=(-\gamma\x,1)$, $(\x,\mathrm{s})\in \mathbb{R}^2_+$ and the flow $\phi$ on $\mathbb{R}^2$ is given by
\[
\phi_t(\x,\mathrm{s})=(e^{-\gamma t}\x,\mathrm{s}+t).
\]
Assumption \ref{a:nons} holds with $m$ being the two-dimensional Lebesgue measure and
$J_t(\x,\mathrm{s})=e^{-\gamma t}$.  We have $E^0=(0,\infty)\times (0,\infty)$,   $\Gamma^{-}=(0,\infty)\times\{0\}$, $\Gamma^{+}=\emptyset$ and  $E=E^0\cup \Gamma^{-}$. The measure $m^{-}$ in Assumption \ref{a:dive} is $\delta_0\times \Leb$, where $\Leb$ is the one-dimensional Lebesgue
measure.

The only possible jumps are when  bursts occur. The amount of molecules is changed  from $\x(\tau_n)$ to $x(\tau_n)+\eta_n$ and we reset the clock. Thus, the jump distribution is 
\[
\jump ((\x,\mathrm{s}), B)=\int_0^\infty \mathbf{1}_B(\x+\mathrm{y},0)h(\mathrm{y})d\mathrm{y},\quad (\x,\mathrm{s})\in E^0.
\]
To calculate the rate of jumps  we observe that
if at time $t$ bursts have not occurred yet, the limiting probability that bursts will occur in the next $\Delta t$ is determined by
\[
\varrho(t)=\lim_{\Delta t \to 0}\frac{1}{\Delta t}\Pr(T_1\in (t,t+\Delta t]|T_1>t)=\frac{h_T(t)}{\int_{t}^\infty h_T(r)dr}.
\]
Hence, $q(\x,\mathrm{s})=\varrho(\mathrm{s})$, $(\x,\mathrm{s})\in \mathbb{R}^2_+$
and $q$ is continuous, implying that Assumption~\ref{a:varp} is satisfied.
We see that condition \eqref{e:jump} holds with $P_0\equiv 0$ and $P_{\partial}\colon L^1(E,m)\to L^1(\Gamma^{-},m^{-})$ being of the form
\[
P_{\partial}(f)(\x,0)=\int_0^\infty \int_0^{\x} f(\mathrm{y},\mathrm{s})h(\x-\mathrm{y})d\mathrm{y}d\mathrm{s},\quad f\in L^1(E,m).
\]
Thus $B\equiv 0$ and the boundary operator $\Psi$ is given by
\[
\Psi(f)(\x,0)=\int_0^\infty \int_0^{\x} \varrho(\mathrm{s})f(\mathrm{y},\mathrm{s})h(\x-\mathrm{y})d\mathrm{y}d\mathrm{s}, \quad \text{if }\varrho f\in L^1(E,m).
\]
It follows from \eqref{e:psilam} that
\[
\Psi(\lambda)f_{\partial}(\x,\mathrm{s})=e^{\gamma \mathrm{s}-\lambda \mathrm{s}-\int_0^{\mathrm{s}}\rho(s-r)dr}f_\partial(e^{\gamma \mathrm{s}}\mathrm{x},0), \quad   f_{\partial}\in L^{1}(\Gamma^{-} ,m^{-}),\lambda>0.
\]
Simple calculations show that
\[
\|\Psi\Psi(\lambda)\|=\int_{0}^\infty e^{-\lambda \mathrm{s}}h_T(\mathrm{s})d\mathrm{s}<1,\quad \lambda>0.
\]
Thus the induced substochastic semigroup is stochastic, by Theorem~\ref{th:ptsmal} and  Corollary~\ref{c:coro1}. Its generator is the operator $A_{\Psi}$, where for sufficiently smooth functions we have
\[
A_{\Psi}f(\x,\mathrm{s})=\frac{\partial}{\partial \x}(\gamma \x f(\x,\mathrm{s}))-\frac{\partial}{\partial \mathrm{s}}( f(\x,\mathrm{s}))-\varrho(\mathrm{s})f(\x,\mathrm{s}),\quad f(\x,0)=\Psi(f)(\x,0).
\]
\end{exampl}

\begin{exampl}[Transport equations on infinite networks]
As in Introduction consider the free transport on
$E^0=(0,1)\times V$, where $V$ is at most a countable subset of $\mathbb{R}_+$ and $\nu$ is the counting measure on $V$. We have
\[
 \Gamma^{-} =\{0\}\times V, \quad \Gamma^{+} =\{1\}\times V, \quad m^{-}(d\x,d\z)=\z\delta_{0} (d\x)\nu(d\z),\quad m^{+}(d\x,d\z)=\z\delta_{1} (d\x)\nu(d\z).
\]
We let $q(\x,\z)=0$ and
\[
\jump((1,\z),B)=\sum_{\z'\in V}\mathbf{1}_B(0,\z')p(\z,\z'),\quad \z\in V, B\subset V,
\]
where we assume that $(p(\z,\z'))_{\z,\z'\in V}$ is a stochastic transition matrix, i.e.
\[
p(\z,\z')\ge 0, \quad \sum_{\z'\in V}p(\z,\z')=1,\quad \z,\z'\in V.
\]
The operator $H\colon L^1(\Gamma^{+},m^{+})\to L^1(\Gamma^{-},m^{-})$ is thus of the form
\[
H(f_{\partial+})(0,\z)=\frac{1}{\z}\sum_{\z'\in V}f_{\partial+}(1,\z')p(\z',\z)\z'
\]
Note that we have
\[
t_{-}(\x,\z)=\frac{\x}{\z},\quad  t_{+}(\x,\z)=\frac{1-\x}{\z}, \quad \x\in [0,1], \z\in V\setminus\{0\}.
\]
Thus, if $V$ is bounded from above then $\inf\{t_{+}(0,\z):\z\in V\}>0$. Consequently, the induced substochastic semigroup is stochastic by Theorem~\ref{th:ptsmal} and Corollary~\ref{c:pert1} with generator being the operator $A_\Psi$
\[
A_{\Psi}f(\x,\z)=-\z\frac{\partial}{\partial \x}f(\x,\z)
\] with $f$ satisfying
\[
\z f(0,\z)=\sum_{\z'\in V}f(1,\z')p(\z',\z)\z'.
\]

Suppose now that $V$ is unbounded and that there exists a quasi-interior element
$\pi$ of $L^1(V,\nu)$ such that
\begin{equation}\label{e:positiverec}
\sum_{\z\in V}\pi(\z)p(\z,\z')\le \pi(\z'),\quad \z'\in V.
\end{equation}
Then
$f_{\partial+}(1,\z)=\pi(\z)/\z$, $\z\in V$, is a quasi interior element of $L^1(\Gamma^{+},m^{+})$ satisfying the assumptions of Corollary~\ref{c:perturb2} and implying that the induced substochastic semigroup is stochastic.
It should be noted that if the transition matrix  $(p(\z,\z'))_{\z,\z'\in V}$ is irreducible then the existence of a quasi-interior element $\pi$ of $L^1(V,\nu)$ satisfying \eqref{e:positiverec} is equivalent to the transition matrix to be positive recurrent, i.e. the transposed matrix is a stochastic operator on $L^1(V,\nu)$ and has a non-zero fixed point. This example can be interpreted as a flow on infinite networks, see \cite{dorn08,dorn10,banasiakfalkiewicz15}.
\end{exampl}

\begin{exampl}[Spatially inhomogeneous linear Boltzmann equations with boundary conditions]  
As in Introduction consider the free transport on
$E^0=\Omega\times V$.
We have
\[
 m^{\pm}(d\x,d\z)=\pm \z\cdot n(\x)\sigma (d\x)\nu(d\z) ,
\]
where $\sigma$ is the surface Lebesgue measure on the boundary $\partial \Omega$.
Let the collision kernel does not depend on $\x$. We  have
\[
Bf(\x,\z)=\int_V \kappa(\mathrm{v},\mathrm{v}')f(\mathrm{x},\mathrm{v}')\nu(d\mathrm{v}'),
\]
and $q$ depends only on $\z$. Let
$\Psi=H\Tr^+$, where $H\colon L^1(\Gamma^{+},m^{+})\to L^1(\Gamma^{-},m^{-})$ is a stochastic operator.
We assume that $\sigma(\partial\Omega)<\infty$, $q$ is strictly positive, and that
there exists a quasi-interior element $\mathcal{M}\in L^1( V, \nu)$ such that
\begin{equation}\label{d:maxwell}
B \mathcal{M}\le q\mathcal{M}, \quad H(\mathcal{M})\le\mathcal{M},
\end{equation}
with $f=q \mathcal{M}$ and $f_\partial=\mathcal{M}$ belonging to $L^1(\Omega\times V,m)$ and $ L^1(\Gamma^{-}, m^{-})$.
It follows from \eqref{d:R0p} that
\[
R_0(f,f_\partial)(\x,\z)\le \mathcal M(\z), \quad (\x,\z)\in (\Omega\times V)\cup \Gamma^{+}.
\]
Hence,
\[
BR_0(f,f_\partial)\le B\mathcal{M}\le f, \quad H(R_0(f,f_\partial))\le H(\mathcal{M})\le f_\partial,
\]
implying that condition \eqref{e:qiP01} holds and that the induced semigroup is stochastic, by Theorem~\ref{th:ptsmal} and Corollary~\ref{c:closure}.
Particular examples of collision kernels for which one can find a Maxwellian function
\[
\mathcal{M}(\z)=\frac{c}{(2\pi\theta)^{d/2}}e^{-\frac{|\z|^2}{2\theta}}
\]
with the above properties are to be found in linear Boltzmann equations with hard potentials  and angular cut-offs, see in particular \cite{lods09,lodsmustapha17}. This example can be extended to problems when the detailed balance condition holds for the kernel $\kappa$ and for the boundary operator, see \cite{pettersson1993weak}.
\end{exampl}

\section*{Acknowledgments}
The author would like to thank  referees for valuable comments which materially improved the presentation of the paper. This research 

%

\end{document}